\documentclass[12pt]{amsart}
\usepackage{amsmath,amsthm,amscd,amssymb,bbm,graphicx}
\usepackage{floatflt,setspace, wrapfig}
\usepackage{showlabels}
\usepackage{graphics,multicol}
\usepackage{mathrsfs} 
\usepackage{epstopdf}
\usepackage{geometry}
\usepackage{hyperref}

\newtheorem{theorem}{Theorem}

\newtheorem{lemma}[theorem]{Lemma}
\newtheorem{proposition}[theorem]{Proposition}
\newtheorem{corollary}[theorem]{Corollary}
\newtheorem{remark}[theorem]{Remark}

\newcommand{\var}{\mbox{\rm var}}

\def\PP{\ensuremath{\mathbb P}}

\def\EEE{\ensuremath{\mathcal E}}

\def\N{\ensuremath{\mathbb N}}
\def\NN{\ensuremath{\mathbb N}}
\def\es{{\emptyset}}

\title {Return times at periodic points in random dynamics}
\date{\today}

\author{Nicolai Haydn}
\address{Mathematics Department\\ University of Southern California\\
Los Angeles, 90089-1113\\  USA}
\email{\href{mailto:nhaydn@usc.edu}{nhaydn@usc.edu}}

\author{Mike Todd}
\address{Mike Todd\\ Mathematical Institute\\
University of St Andrews\\
North Haugh\\
St Andrews\\
KY16 9SS\\
Scotland} \email{\href{mailto:m.todd@st-andrews.ac.uk}{m.todd@st-andrews.ac.uk}}
\urladdr{\url{http://www.mcs.st-and.ac.uk/~miket/}}

\begin{document}
\maketitle

\begin{abstract}
We prove a quenched limiting law for random measures on subshifts 
at periodic points. We consider a family of measures $\{\mu_\omega\}_{\omega\in\Omega}$, 
where the `driving space' $\Omega$ is equipped with a probability measure
which is invariant under a transformation $\theta$.   
We assume that the fibred measures $\mu_\omega$  satisfy a generalised invariance 
property and are $\psi$-mixing. We then show that for almost every $\omega$
the return times to cylinders $A_n$ at periodic points are in the limit 
compound Poisson distributed for a parameter $\vartheta$ which is given by 
the escape rate at the periodic point.

\end{abstract}

\section{Introduction}
\label{sec:intro}

For sufficiently mixing deterministic dynamical systems the return times to shrinking sets (e.g.~dynamical cylinders) around a typical point in the phase space are in the 
limit exponentially distributed almost surely, as shown in
\cite{Aba06} and~\cite{AS11}. Moreover, for $\phi$-mixing measures it 
follows from~\cite{AV09} that for all non-periodic points one obtains in the
limit the exponential distribution for entry and return times, but 
that at periodic points the limiting return times distribution have
a point mass at the origin. A similar distinction can be drawn for
higher order returns where we know that for $\psi$-mixing
systems return times at periodic points 
are in the limit compound Poisson distributed~\cite{HV09}. 
Assuming the $\phi$-mixing property we can again conclude
 that higher order return times are in the limit Poisson distributed
 (\cite[Corollary~1]{HP14}).
 
 For random, stochastic dynamical systems, following work in \cite{RouSauVar13} for typical points, it was shown in~\cite{RT15} that 
 the entry times distributions at periodic points show similar behaviour as in
 the deterministic setting if one considers a quenched limit  (the annealed result then follows easily). In this case 
 the limiting distribution has a point mass at the origin and is otherwise 
 exponential. The relative weight of the point mass is determined by the marginal measure
 and applies to almost all realisations of the random dynamics. It is assumed that the marginal
 measures are $\psi$-mixing. In the present paper we consider the same 
 setting and prove that higher order return times at periodic points
 are compound Poisson distributed. Again the parameter for the 
 compound Poissonian is entirely determined by the marginal measure
 and applies to almost all realisations.
 
 The perspective taken here and in the works discussed above is to see the system via dynamically defined cylinder sets, which makes it essentially a `symbolic approach'.  We note that outside this context (for example for an interval map, considering hits to balls rather than cylinders), much less is known in the random setting.  However, in \cite{AytFreVai14}, a Poisson distribution was shown for first hitting times to balls in the setting of certain random dynamical systems.  We note that this was for systems which were all close to a certain well-behaved system, so the randomness could be interpreted as (additive) noise.  Moreover, this was an annealed law rather than a quenched one.

 \subsection{Setting and conditions}

For a measure space $\Omega$, let $\theta:\Omega\to \Omega$ be an invertible transformation preserving an ergodic probability measure  $\PP$.  Set $\Sigma=\NN^{\NN_0}$ and let
$\sigma: \Sigma \to \Sigma$ denote the shift.
 Let $A=\left\{A(\omega)=(a_{ij}(\omega)):\omega\in \Omega\right\}$ be a random transition matrix, i.e., for any $\omega\in\Omega$, $A(\omega)$ is a $\N\times \N$-matrix with entries in $\{0,1\}$ such that $\omega\mapsto a_{ij}(\omega)$ is measurable for any $i\in\NN$ and $j\in\NN$.  
For any $\omega\in \Omega$
we write
$$
\Sigma_\omega =\{x=(x_0,x_1,\ldots)\colon x_i\in \N \text{ and } a_{x_i x_{i+1}}(\theta^i\omega)=1\text{ for all } i\in\NN\}
$$
for the fibre over the point $\omega$ in the `driving space' $\Omega$.
Moreover
$$
\EEE = \{(\omega,x)\colon \omega\in\Omega,x\in \Sigma_\omega\} \subset \Omega\times \Sigma
$$
denotes the full space on which the random system is described through a skew action.
Indeed, the random dynamical system coded by the skew-product $S : \EEE \to \EEE$ given by
$S(\omega,x)= (\theta \omega,\sigma x)$.
While we allow infinite alphabets here, we nevertheless call $S$ a \emph{random subshift of finite type} (SFT).  Assume that $\nu$ is an $S$-invariant probability
measure with marginal $\PP$ on $\Omega$.  Then we let $(\mu_\omega)_\omega$ denote
its decomposition on $\Sigma_\omega$ (see \cite[Section 1.4]{Arn98}), that is, $d\nu(\omega,x)=d\mu_\omega(x)d\PP(\omega)$. 
The measures $\mu_\omega$ are called the \emph{sample measures}.  Observe that $\mu_\omega(A)=0$ if $A\cap \Sigma_\omega=\es$. We denote by $\mu=\int \mu_\omega \, d\PP$ the marginal of $\nu$ on $\Sigma$.  We note that we may replace the assumption of invertibility of $\theta$ by assuming the existence of sample and marginal measures as above.

We also identify our alphabet $\mathcal{A}$ with the partition given by $1$-cylinders 
$U(a)=\{x\in\Sigma: x_0=a\}$. The elements of the $k$th join 
$\mathcal{A}^k=\bigvee_{j=0}^{k-1}\sigma^{-j}\mathcal{A}$, $k=1,2,\dots$ are called
 {\em $k$-cylinders}. Put $\mathcal{A}^*$ for the forward sigma-algebra
generated by $\bigcup_{j\ge1}\mathcal{A}^j$. The length $|A|$ of a cylinder set $A$ is
determined by $|A|=k$ where $k$ is so that $A\in\mathcal{A}^k$.
Note that $\mathcal{A}$ is generating, i.e.\ that the atoms of 
$\mathcal{A}^\infty$ are single points.
If we denote by $\chi_A$ the characteristic function of a (measurable) set $A\subset\Sigma$
 then we can define the counting function
$$
\zeta_A(z)=\sum_{j=1}^{N}\chi_A\circ \sigma^j(z),
$$
$z\in \Sigma$, where $N$ is the observation time given by the invariant annealed measure
$\mu$. To wit $N=[t/\mu(A)]$ for $t>0$ a parameter.
The value of $\zeta_A$ counts the number of times a given point returns to $A$ within the time
$N$. 

Let us make the following assumptions: \\

\noindent (i) The measures $\mu_\omega$ are $\psi$-mixing: There exists a decreasing function
$\psi:\mathbb{N}\to[0,\infty)$ (independent of $\omega$) so that 
$$
\left|\mu_\omega(A\cap\sigma^{-n-k}B)-\mu_\omega(A)\mu_{\theta^{n+k}\omega}(B)\right|
\le\psi(k)\mu_\omega(A)\mu_{\theta^{n+k}\omega}(B)
$$ 
for all $A\in\sigma(\mathcal{A}^n)$ and $B\in\mathcal{A}^*$.\\
(ii) The marginal measure $\mu$ satisfies the $\alpha$-mixing property:
$$
\left|\mu(A\cap \sigma^{-n-k}B)-\mu(A)\mu(B)\right|\le \psi(k)
$$ 
for all $A\in\sigma(\mathcal{A}^n)$ and $B\in\mathcal{A}^*$.\\
(iii) There exist $0<\eta_0<1$ so that $\eta_0^n\le\mu(A)$ for 
all $A\in\mathcal{A}^n$, all $\omega$ and all large $n$.\\
(iv) 
$$
\sup_\omega\sup_{A\in\mathcal{A}}\mu_\omega(A)< 1.
$$

\vspace{3mm}

\noindent  Our main result, Theorem~\ref{psi-mixing}, is that under these conditions, the return times at periodic points $x$ are 
compound Poissonian provided the limit $\vartheta(x)=\lim_{n\to\infty}\frac{\mu(A_{n+m}(x))}{\mu(A_n(x))}$
exists, where $m$ is the minimal period of $x$ and $A_n(x)\in\mathcal{A}^n$ denotes the
$n$-cylinder that contains $x$.
To be more precise, if we denote by $\zeta_n^x$  the counting function
$$
\zeta_n^x(z)=\sum_{j=1}^{N_n} \chi_{A_n(x)}\circ \sigma^j(z)
$$ 
with the {\em observation time}
$N_n=\left[\frac{t}{\mu(A_n(x))}\right]$ ($t>0$ is a parameter), then we will show that
$\mu_\omega\!\left(\zeta_n^x=r\right), r=0,1,2,\dots$, converges to the Polya-Aeppli
distribution as $n\to\infty$ for $\mathbb{P}$-almost every $\omega$.

The first such result was by Hirata~\cite{Hirata1} for the first entry time for Axiom~A systems.
For random systems satisfying  assumptions~(i)--(iv) a similar result was then shown
by Rousseau and Todd~\cite{RT15} for the first entry time distribution in the quenched case. 
  Note that, as mentioned above, for systems perturbed by additive noise, which are a particular case of our systems here, an annealed version of this result is proved in  \cite{AytFreVai14}.  The additivity `washes out' any periodic behaviour.

As in \cite{RT15}, if we wish to consider shifts on countable alphabets, it is no longer reasonable to assume condition (iii), but if we drop this and strengthen condition (ii) to the assumption of the $\psi$-mixing property for $\mu$, then our results still follow.  We close this section by noting that the conditions on our systems here are the same as those in  \cite{RT15}, so the main result here also applies to all the applications given there.

\subsection{Structure of the paper}

 In Section~\ref{factorial.moments} we describe the compound Poisson
 distribution and state an auxiliary limiting result on which the proof of 
 the main result is based. The main part of the proof of the main result
 consists of estimating the contributions made by short returns which 
 are outside the periodicity of the periodic point and for which the 
 mixing property cannot be well applied. This is done in Section~\ref{rare.sets}.
 The compound part is 
 determined by the periodic behaviour near the periodic point where
 it generates geometrically distributed immediate returns. For the long returns
 the mixing property comes into play and results in exponentially distributed
 returns between clusters of short returns whose numbers are in the limit
 geometrically distributed.
 We state our main result in Section~\ref{distribution} and provide examples in
 Section~\ref{examples}.

\textit{Acknowledgements.}  This work was begun at the Conference on Extreme Value Theory and Laws of Rare Events at Centre International de Rencontres Math\'ematiques (CIRM), Luminy and part of it continued at the American Institute of Mathematics (AIM).  The authors thank both institutions for their hospitality.

\section{Factorial moments and a limiting result}\label{factorial.moments}
This section is used to recall a result on the approximation of the compound 
Poisson distribution with a geometric distribution, i.e.\ the Polya-Aeppli 
distribution. More general compound Poisson distributions were considered in~\cite{Fel}
and more  recently (e.g.~\cite{CR,BCL}) there have
been efforts to approach compound Poisson distributions using the Chen-Stein method.
Although the treatment in~\cite{CR} applies to more general setting, the result is far from 
applicable to our situation. Proposition~\ref{sevastyanov} is the compound 
analogue of other theorems for the plain  Poisson distribution as for instance in~\cite{Sev,HV}.

\subsection{Compound Poisson distribution}

For a parameter $p\in[0,1)$ define the polynomials
$$
P_r(t,p)=\sum_{j=1}^rp^{r-j}(1-p)^j\frac{t^j}{j!}\binom{r-1}{j-1},
$$
$r=1,2,\dots$, where $P_0=1$ ($r=0$). The distribution $e^{-t}P_r(t,p)$, $r=0,1,2,\dots$ 
is  the {\em P\'olya-Aeppli distribution}~\cite{JKW} which has the generating function
$$
g_p(z)=e^{-t}\sum_{r=0}^\infty z^rP_r=e^{t\frac{z-1}{1-pz}}.
$$
It has mean $\frac{t}{1-p}$ and variance $t\frac{1+p}{(1-p)^2}$. 
For $p=0$: $e^{-t}P_r(t,0)=e^{-t}\frac{t^r}{r!}$ and one recovers the Poisson distribution whose 
 generating function $g_0(z)=e^{t(z-1)}$ is analytic in the entire plane whereas for $p>0$ the generating function $g_p(z)$ has an essential singularity at $\frac1p$. 
 The expansion at $z_0=1$ yields
$g_p(z)=\sum_{k=0}^\infty (z-1)^kQ_k$ where
$$
Q_k(t,p)=\frac1{(1-p)^k}\sum_{j=1}^kp^{k-j}\frac{t^j}{j!}\binom{k-1}{j-1}
$$
($Q_0=1$) are the factorial moments. 

\subsection{Return times patterns}\label{returntimespatterns}
Let $M$ and $m<M$ be given integers (typically $m<\!\!<M$) and let $N\in\mathbb{N}$
 be some (large) number. For $r=1,2,3,\dots$ we define the following:\\
{\bf (I)} $G_r(N)$: We denote by $G_r(N)$ the simplex of $r$-vectors
$\vec{v}=(v_1,\dots,v_r)\in\mathbb{N}^r$ for which $1\leq v_1<v_2<\cdots<v_r\leq N$.\\
{\bf (II)} $G_{r,j}(N)$: We write $G_r$ as the disjoint union $\bigcup_{j}G_{r,j}$
where $G_{r,j}$ consists of all $\vec{v}\in G_r$ for which we can find $j$ indices
$i_1,i_2,\dots,i_j\in\{1,2,\dots,r\}$, $i_1=1$, so that $v_k-v_{k-1}\le M$ if
$k\not=i_2,\dots,i_j$ and so that $v_k-v_{k-1}>M$ for all $k=i_2,\dots,i_j$. 

For $\vec{v}\in G_{r,j}$ the values of 
$v_i$ will be identified with returns; returns that occur within less than time $M$ are called
{\em immediate returns} and if the return time is $\ge M$ then we call it a {\em long return}
(i.e.\ if $v_{i+1}-v_i<M$ then we say $v_{i+1}$ is an immediate return and if $v_{i+1}-v_i\ge M$
the we call $v_{i+1}$ a long return). 
That means that  $G_{r,j}$ consists of all return time patterns $\vec{v}$ which have $r-j$
immediate returns that are clustered into $j$ blocks of immediate returns and $j-1$ long returns between those blocks. The entries $v_{i_k}$, $k=1,\dots,j$, are the beginnings (heads) of the blocks
(of immediate returns). We assume from now on that all short returns are multiples of $m$. 
(This reflects the periodic structure around periodic points as in condition~(b) of 
Proposition~\ref{sevastyanov}.)\\
{\bf (III)} $G_{r,j,w}(N)$:  For $\vec{v}\in G_{r,j}$ the length of each block is 
$v_{i_{k+1}-1}-v_{i_k}$, $k=1,\dots,j-1$. Consequently 
let us put $w_k=\frac1m(v_k-v_{k-1})$ for the {\em individual overlaps}, for $k\not=i_1,i_2,\dots,i_j$. Then
$\sum_{\ell=i_k+1}^{i_{k+1}-1}w_\ell=\frac1m(v_{i_{k+1}-1}-v_{i_k})$ is the {\em overlap} of the
 $k$th block and
$w=w(\vec{v})=\sum_{k\not=i_1,i_2,\dots,i_j}w_k$ the {\em total overlap} of $\vec{v}$.
If we put
$G_{r,j,w}=\{\vec{v}\in G_{r,j}: w(\vec{v})=w\}$ then $G_{r,j}$ is the disjoint union $\bigcup_{w}G_{r,j,w}$.\\
{\bf (IV)} $\Delta(\vec{v})$: For $\vec{v}$ in $G_{r,j}$ we put
$$
\Delta(\vec{v})=\min\left\{v_{i_k}-v_{i_k-1}:\;k=2,\dots,j\right\}
$$
for the minimal distance between the `tail' and the `head' of successive blocks of immediate returns.

\subsection{Compound Poisson approximations}

We shall use the following result:

\begin{proposition}\label{sevastyanov} Let $m$ be as above and assume that 
there are sequences $\{M(n):n\},\{N(n):n\}$ and $0,1$-valued random variables
 $\rho_{j,n}$ for  $j=1,\dots,N(n)$ on some $\Sigma$. 
For $\vec{v}\in G_r(n)$ put $\rho_{\vec{v}}=\prod_i\rho_{v_i,n}$.
Choose $\delta(n)>0$ and define the `rare set' $R_r(n)=\bigcup_{j=1}^rR_{r,j}$,
 where $R_{r,j}=\{\vec{v}\in G_{r,j}: \Delta(\vec{v})<\delta\}$.
Let $\mu$ be a probability measure on $\Sigma$ which satisfies the following conditions: \\
(a)
$$
\sum_{\vec{v}\in G_r\setminus R_r}\mu(\rho_{\vec{v}})
\to Q_r(t,p)
$$
as $n\to\infty$ for some $p\in(0,1]$.\\
(b) 
$$
\sum_{\vec{v}\in R_r}\mu(\rho_{\vec{v}})
\to 0 
$$
as $n\to\infty$.

Then for every $r$
$$
\mu(\zeta_n=r)\to e^{-t}P_r(t,p)
$$
as $n\to\infty$, where $\zeta_n=\sum_{j=1}^{N(n)}\rho_{j,n}$.
\end{proposition}

\begin{proof} 
The result follows by the moment method (see for example \cite[Section 30]{Bil12}) that 
$\mu(\zeta_n^{(r)})$ converges to $Q_r(t,p)$ for each $r$, 
where $\zeta_n^{(r)}=\zeta_n(\zeta_n-1)\cdots(\zeta_n-r+1)$ is the factorial
moment. Since $\mu(\zeta_n^{(r)})=\sum_{\vec{v}\in G_r}\mu(\rho_{\vec{v}})$
we obtain by  assumptions~(a) and~(b) that
$$
\mu(\zeta_n^{(r)})=\sum_{\vec{v}\in G_r\setminus R_r}\mu(\rho_{\vec{v}})
+\sum_{\vec{v}\in  R_r}\mu(\rho_{\vec{v}})\to Q_r(t,p)
$$
since the second term goes to zero and the first term converges to $Q_r$.
\end{proof}

\vspace{3mm}

\noindent In the following we will apply this proposition to situations that typically
arise in dynamical systems.  There the stationarity condition~(a) of the proposition
is implied by the invariance of the measure. The random variables $\rho_j$ will be
the indicator function of a cylinder set pulled back under the $j$th iterate of the map. 
Condition~(a) is then implied by the mixing property.
The more difficult condition to satisfy is~(b) because it involves `short range' interaction
over which one has little control and which requires more delicate estimates
(see Lemma~\ref{R.small} below).

\section{$\psi$-mixing measures and the rare set}\label{rare.sets}


In this section we only assume Assumption~(i), that is the measures $\mu_\omega$ are 
$\psi$-mixing i.e.\ satisfy
$$
\left|\mu_\omega(U\cap \sigma^{-m-n}V)-\mu_\omega(U)\mu_{\theta^{m+n}\omega}(V)\right|\leq\psi(m)\mu_\omega(U)\mu_{\theta^{m+n}\omega}(V)
$$
for all $U\in\sigma(\mathcal{A}^n)$, $V\in\sigma(\mathcal{A}^*)$ and for all $m, n\ge0$,
where $\psi(m)\to0$ (and $\psi$ is independent of $\omega$).

 For instance  equilibrium states for H\"older continuous
potentials on Axiom A systems (which include subshifts of finite type) or on the 
Julia set of hyperbolic rational maps are $\psi$-mixing~\cite{DU1}.

\vspace{2mm}

\noindent For $r\geq1$ and (large) $\tau\in\mathbb{N}$ let as above $G_r(N)$ be the $r$-vectors
$\vec{v}=(v_1,\dots,v_r)\in\mathbb{Z}^r$ for which $1\leq v_1<v_2<\cdots<v_r\leq N$. 
 Let $t$ be a positive parameter, $W\subset \Sigma$ and put $\tau=[t/\mu(W)]$ be the normalised time. 
 Then the
entries $v_j$ of the vector $\vec{v}\in G_r(N)$ are the times at which all the points in
$C_{\vec{v}}=\bigcap_{j=1}^r \sigma^{-v_j}W$ hit the set $W$ during the time interval $[1,N]$. 
 The following lemma, a random version of \cite[Lemma 4]{HV},  is immediate.
\begin{lemma}\label{product.mixing}
Let $(\sigma,\mu_\omega)$ be $\psi$-mixing. 
For  $r>1$ let $n_i\ge1, i=1,\dots,r-1$, be given numbers and $\vec{n}=(n_1,  \ldots, n_{r-1})$. Let $W_i\in\sigma(\mathcal{A}^{n_i})$
and assume that $\vec{v}\in G_r(N)$ is such that $v_{i+1}-v_i\geq n_i$ ($i=1,\dots,r-1$).
Then
$$
\left|\frac{\mu_\omega\left(\bigcap_{i=1}^r\sigma^{-v_i}W_i\right)} 
{\prod_{i=1}^r\mu_{\theta^{v_i}\omega}(W_i)}-1\right|
\leq(1+\psi(d(\vec{v},\vec{n})))^{r-1}-1,
$$
and $d(\vec{v},\vec{n})=\min_{1\le i\le r-1}(v_{i+1}-v_i-n_i)$.
\end{lemma}

\begin{remark}
As in \cite[Lemma 2.1]{RT15}, and similarly to the above lemma, under conditions (i) and (iv) there exists  $0<\eta_1<1$ so that for all large $n$
$$\mu_\omega(A)\le \eta_1^n$$ for 
all $A\in\mathcal{A}^n$, and $\PP$-a.e. $\omega$.
\label{rmk:eta}
\end{remark}

\subsection{Estimate of the rare set}\label{section.return.times}
Next we will estimate the size of the rare set. 
As before we put $C_{\vec{v}}=\bigcap_{k=1}^r\sigma^{-v_k}W$  for $\vec{v}\in G_r(N)$  where $W\in\sigma(\mathcal{A}^n)$ for some $n$. Let $\delta\geq0$ and put 
$$
R_{r,j}(N)=\{\vec{v}\in G_{r,j}(N): \min_k(v_{i_k+1}-v_{i_k}-n)<\delta\},
$$
where the values $v_{i_1},\dots,v_{i_j}$ are the beginnings of the $j$ blocks of
immediate returns (notation as in section~\ref{returntimespatterns}~(II)).
Then we put $R_r=\bigcup_jR_{r,j}$.

\begin{lemma}\label{R.small} 
Let the class of measures $\mu_\omega$ be $\psi$-mixing.  Let $\{A_n\in \mathcal{A}_n:n\}$ 
be a sequence of cylinders and $\{M_n<n:n\}$ a sequence of integers so that for all large $n$,  $A_n\cap\sigma^{-\ell}A_n\not=\emptyset$ for
$\ell<M=M_n$ implies that $\ell$ is a multiple of some given integer $m$.

Then there exists a constant $K_1$ so that
$$
\sum_{\vec{v}\in R_r} \mu_\omega(C_{\vec{v}}) \leq K_1\gamma^{r-1}\sum_{j=2}^r\sum_{s=1}^{j-1}
\binom{j-1}{s-1} (\delta\eta_1^M)^{j-s}\frac{\mu_\omega(\zeta_n^x)^s}{s!}
\binom{r-1}{j-1}(\alpha'\eta_1^m)^{r-j},
$$
for $\delta=\delta_n>n$ and $R_r$ as above, where $C_{\vec{v}}=\bigcap_{k=1}^r\sigma^{-v_k}A_n$,  $\gamma=1+\psi(\delta-n)$ 
and $\alpha'>1+\psi(0)$.
\end{lemma}

\begin{proof} Put $R_{r,j}^s$ for those $\vec{v}\in R_{r,j}$ for which $v_{i+1}-v_i\geq\delta$ for
$s-1$ indices $i_1, \dots, i_{s-1}$ and $i_s=r$ ($s\le j-1$) indicate the tails of the blocks
which are followed by a large gap. Similarly we put $R_{r,j,u}^s$ for the set $R_{r,j}^s\cap R_{r,j,u}$.
 We consider two separate cases: (A) $s\geq2$ and (B) $s=1$. 

\noindent {\bf (A)} Assume  $s\geq2$ and $i'_1, i'_2, \dots, i'_{j}$ be the $j$ tails of 
blocks ($i'_j=r$) which are characterised by $v_{i'_k+1}-v_{i'_k}\geq M$ for $k=1,\dots,j-1$
 (and $v_{i'_j}=v_r$). We have $\{i_k:k\}\subset\{i'_k:k\}$ where the $j-s$ many indices
 in $\{i'_k:k\}\setminus\{i_k:k\}$ mark the gaps which are $\geq M$ and smaller than $\delta$.
 Moreover, the remaining  $r-j$ return times are immediate short returns of lengths $\in[m,M)$. 
 Let us put
\begin{eqnarray*}
W_{i_k'}&=&A_{m}\cap \sigma^{-m}A_{n}\cap\sigma^{-2m}A_{n}
\cap \cdots\cap\sigma^{-(u_k-1)m}A_{n}\cap\sigma^{-u_km}A_{n}\\
&=&A_{m}\cap \sigma^{-m}A_{m}\cap\sigma^{-2m}A_{m}
\cap \cdots\cap\sigma^{-(u_k-1)m}A_{m}\cap\sigma^{-u_km}A_{n},
\end{eqnarray*}
where $u_k$ is the overlap for the $k$th block. The $\psi$-mixing property yields
the following estimate
$$
\mu_\omega(W_{i'_k})
\le\alpha_2^{u_km}\left(\prod_{\ell=0}^{u_k}\mu_{\theta^{\ell m}\omega}(A_m)\right)
\mu_{\theta^{u_km}\omega}(A_n)
\le(\alpha_2\eta_1^m)^{u_k}\mu_{\theta^{u_km}\omega}(A_n),
$$
where $\alpha_2=1+\psi(0)$ and where we used that  (see the Remark~\ref{rmk:eta})
 $\mu_{\omega'}(A_m)\le\eta_1^m$
for any $\omega'$. By Lemma~\ref{product.mixing}
\begin{eqnarray*}
\mu_\omega\left(C_{\vec{v}}\right)
&\leq&\mu_\omega\left(\bigcap_{i=k}^{j}\sigma^{-(v_{i'_k}-u_k)}W_{i'_k}\right)\\
&\leq&\gamma^{s-1}\alpha_2^{j-s}\prod_{i=1}^{j}\mu_{\theta^{v_{i'_k}-u_k}\omega}(W_{i'_k})\\
& \leq&\gamma^{s-1}\alpha_2^{j-s}(\eta_1^M)^{j-s}(\alpha_2\eta_1^m)^u\prod_{k=1}^s\mu_{\theta^{v_{i_k}}\omega}(A_n),
\end{eqnarray*}
where $\gamma=1+\psi(\delta-n)$, 
and the components of $\vec{n}=(n_1,\dots,n_r)$ as in Lemma~\ref{product.mixing} are
given by $n_{i_k}=n$ for $k=1,\dots,s$ (for the long returns between clusters, i.e., $>\delta$) and $n_i=M$ for $i\not=i_k,\;k=1,\dots,s$, where 
$u=\sum_i u_i$ is the total overlap.  We have used that  $\mu_\omega(A_M)\le\eta_1^M$ for any $\omega$.

To count the number of return times vectors, note that there are 
$\binom{r-1}{j-1}$ many possibilities to choose the 
$j$ positions $i'_1,\dots,i'_j$ of the returns $>M$. Out of those we can pick in
$\binom{j-1}{s-1}$ many ways the long return
($\ge\delta$) positions $i_1,\dots,i_s$. Moreover, each choice allows for $<\delta^{j-s}$
many ways to fill in the actual $j-s$ many intermediate return times (between $M$ and $\delta$).

For every fixed set of $j$ returns larger than  $M$ and for a fixed value 
of overlaps $u$ there are 
$\binom{u-1}{r-j-1}$ many ways to distribute the $u$ overlaps 
into the remaining $r-j$ many returns which are shorter than $M$.

For each fixed set of long ($\ge\delta$) return times $v_{i_1},\dots,v_{i_s}$ and given 
value of overlaps $u$ there are consequently 
$$
 \binom{j-1}{s-1}\binom{r-1}{j-1}\delta^{j-s}\binom{u-1}{r-j-1}
$$
many possibilities. We thus obtain:
\begin{eqnarray*}
\sum_{\vec{v}\in R^s_{r,j,u}} \mu_\omega(C_{\vec{v}})\hspace{-2cm}\\
&\le&\binom{j-1}{s-1}\binom{r-1}{j-1}\delta^{j-s}\binom{u-1}{r-j-1}
\gamma^{s-1}(\eta_1^M)^{j-s}(\alpha_2\eta_1^m)^u
\sum_{v_{i_1}<\dots<v_{i_s}\le N}\prod_{k=1}^s\mu_{\theta^{v_{i_k}}\omega}(A_n)\\
&\le&\binom{j-1}{s-1}\binom{r-1}{j-1}\delta^{j-s}\binom{u-1}{r-j-1}\gamma^{s-1}(\eta_1^M)^{j-s}(\alpha_2\eta_1^m)^u
\frac1{s!}\left(\sum_{i=1}^N\mu_{\theta^i\omega}(A_n)\right)^s.
\end{eqnarray*}
Therefore, since $\mu_\omega(\zeta_n^x)=\sum_{i=1}^N\mu_{\theta^i\omega}(A_n)$,
$$
\sum_{\vec{v}\in R^s_{r,j,u}} \mu_\omega(C_{\vec{v}})
\leq\gamma^{s-1}\binom{j-1}{s-1}\binom{r-1}{j-1}(\delta\eta_1^M)^{j-s} 
\binom{u-1}{r-j-1}(\alpha_2\eta_1^m)^u\frac{\mu_\omega(\zeta_n^x)^s}{s!}.
$$

\noindent {\bf (B)} If $s=1$ then all returns between blocks are less than $\delta$ for all $k$. In the
same way as above we obtain
$$
\sum_{\vec{v}\in R^1_{r,j,u}} \mu_\omega(C_{\vec{v}}) \leq(\delta\eta_1^M)^{j-1}
\binom{r-1}{j-1}
\binom{u-1}{r-j-1}(\alpha_2\eta_1^m)^u\mu_\omega(\zeta_n^x).
$$

\vspace{2mm}

\noindent Summing over $s$ and using the estimates from~(A) and~(B) yields
\begin{eqnarray*}
\sum_{\vec{v}\in R_r}\mu_\omega(C_{\vec{v}}) 
&=&\sum_j\sum_{s=1}^{j-1}\ \sum_{u=r-j}^\infty \ \sum_{\vec{v}\in R_{r,j,u}^s}\mu_\omega(C_{\vec{v}})\\
&\leq&\sum_{j=2}^r\gamma^{s-1} \sum_{s=1}^{j-1} \binom{j-1}{s-1}
\frac{\mu_\omega(\zeta_n^x)^s}{s!}(\delta\eta_1^M)^{j-s} \binom{r-1}{j-1}
\sum_{u=r-j}^\infty\binom{u-1}{r-j-1}
(\alpha_2\eta_1^m)^u\\
&\leq&\sum_{j=2}^r\gamma^{s-1} \sum_{s=1}^{j-1} \binom{j-1}{s-1}
\frac{\mu_\omega(\zeta_n^x)^s}{s!}(\delta\eta_1^M)^{j-s} \binom{r-1}{j-1}
\left(\frac{\alpha_2\eta_1^m}{1-\alpha_2\eta_1^m}\right)^{r-j}\\
\end{eqnarray*}
(as $\sum_{u=q}^\infty\binom{u-1}{q-1}x^u=\left(\frac{x}{1-x}\right)^q$).
The lemma now follows since 
$\frac{\alpha_2\eta_1^m}{1-\alpha_2\eta_1^m}\le\alpha'\eta_1^m$ with an
$\alpha'$ slightly larger than $\alpha_2$. 
\end{proof}

\section{Distribution near periodic points for $\psi$-mixing measures}\label{distribution}

 We will also need the almost sure convergence of 
$\zeta_n^x= \sum_{j=1}^{N_n}\chi_{A_n}\circ\sigma^j$ (where $N_n=\left[t/\mu(A_n(x))\right]$)
 which is proved in \cite[Lemma 9]{RouSauVar13}.  
The following lemma requires the Assumption~(iii).  

\begin{lemma}\label{almost.sure.limit}
If there is $q>2\frac{\log \eta_1}{\log\eta_0}$ such that $\psi(k)k^q\to 0$ as $k\to \infty$, 
then $\mu_\omega(\zeta_n^x)\to t$  for $\PP$-almost every $\omega$.
\end{lemma}

Let us put $\zeta_{n,u}^x=\sum_{k=0}^{N}\chi_{A_{n+mu}}\circ\sigma^k$, where 
$N=\frac{t}{\mu(A_n)}$.  We will assume that the limit
\begin{align}
\vartheta(x)=\lim_{n\to\infty}\frac{\mu(A_{n+m}(x))}{\mu(A_n(x))}
\label{eq:vartheta}
\end{align}
 exists.
Then $\mu(\zeta_{n,u}^x)=N\mu(A_{n+mu})=\frac{\mu(A_{n+mu})}{\mu(A_n)}t$
converges to $\vartheta^ut$ as $n\to\infty$. 
By the same argument as in Lemma~\ref{almost.sure.limit} we conclude the following result
of which Lemma~\ref{almost.sure.limit} is the special case $u=0$.

 \begin{corollary} 
 If $\psi(k)k^q\to 0$ as $k\to \infty$ for some $q>2\frac{\log \eta_1}{\log\eta_0}$ and the limit $\vartheta(x)=\lim_{n\to\infty}\frac{\mu(A_{n+m}(x))}{\mu(A_n(x))}$
exists, then 
$$
\mu_\omega(\zeta_{n,u})\to\vartheta^ut
$$
as $n\to\infty$ for $\PP$-almost every $\omega$. 
\end{corollary}


Although for a periodic point $x$  with (minimal) period $m$ the limit
$\lim_{\ell\rightarrow\infty}\frac1\ell\left| \log\mu(A_{\ell m}(x))\right|$
always exists (see Lemma~7 of~\cite{HV09}), we cannot necessarily conclude that the limit 
$\vartheta=\lim_{n\rightarrow\infty}\frac{\mu(A_{n+m}(x))}{\mu(A_n(x))}$ exists. 

 For $t>0$ and
integers $n$ we put $\zeta_n^t$ for the counting function 
$\sum_{j=0}^{N_n} \chi_{A_n(x))}\circ \sigma^j$ with the observation time
$N_n=\left[t/\mu(A_n(x))\right]$
(where $x$ is periodic with minimal period $m$). For equilibrium states for H\"older 
continuous potentials $f$ (with zero pressure) on Axiom~A systems, Hirata~\cite{Hirata1} has shown
that $\vartheta(x)=\exp\sum_{j=1}^mf(\sigma^jx)$ for  periodic points $x$ with 
minimal period $m$, see Example~\ref{example.Axiom.A}.

\vspace{3mm}

\noindent In order to satisfy the assumptions of Proposition~\ref{sevastyanov} we put
$\gamma=\alpha$,
$\gamma_1=\alpha\delta_n\eta^M$ and $\gamma_2=\alpha\eta^m$

\begin{theorem}\label{psi-mixing} 
Suppose that we have a random SFT driven by an invertible ergodic measure preserving system $(\Omega, \theta, \PP)$ with marginal measure $\mu$ and satisfying conditions (i)--(iv), where the function $\phi$ is such that there is $q>2\frac{\log \eta_1}{\log\eta_0}$ with $\psi(k)k^q\to 0$ as $k\to \infty$.
 Let $x\in \Sigma$ a periodic point with minimal period $m$ and 
assume the limit defining $\vartheta$ exists and let $\Theta=1-\vartheta$.

Then 
$$
\mu_\omega(\zeta_n^x=r)\longrightarrow e^{-t}P_r(\Theta t,\vartheta)
$$
as $n\to\infty$ for $\PP$-a.e. $\omega$.
\end{theorem}

\begin{proof} We use Proposition~\ref{sevastyanov} and have to verify conditions~(a) and~(b). 

Let $\vec{v}\in G_{r,j}\setminus R_{r,j}$ and let $v_{i_1},v_{i_2},\dots, v_{i_j}$ be the 
heads of the blocks of short returns, that is $i_1=1$ and $v_{i_k}-v_{i_k-1}\ge\delta$
for $k=2,\dots,j$. Moreover $v_{\ell}-v_{\ell-1}\le M$ for $\ell\not\in\{i_1,\dots,i_j\}$.
By Lemma~\ref{product.mixing}  
($\Delta$ is as defined in section~\ref{returntimespatterns})
$$
\left|\mu_\omega(C_{\vec{v}})-\prod_{k=1}^j\mu_{\theta^{v_{i_k}}\omega}(A_{n+mu_k})\right|
\le((1+\psi(\Delta(\vec{v})-n))^j-1)\prod_{k=1}^j\mu_{\theta^{v_{i_k}}\omega}(A_{n+mu_k}),
$$
where $u_k$ is the overlap of the $k$th block beginning with $v_{i_k}$. Notice that 
for the $k$th cluster one has
$$
A_{n+mu_k}=\bigcap_{\ell=i_k}^{i_{k+1}-1}\sigma^{-v_\ell}A_n.
$$
Thus
$$
\mu_\omega(C_{\vec{v}})
=(1+\mathcal{O}(j\psi(\delta)))\prod_{k=1}^j\mu_{\theta^{v_{i_k}}\omega}(A_{n+mu_k})
$$
and summing over the total overlaps (see~(III)) yields for the principal term
$$
S_r(n):=\sum_{\vec{v}\in G_r\setminus R_r}\mu_\omega(C_{\vec{v}})
=\sum_{j=1}^r(1+\mathcal{O}(j\psi(\delta)))
\sum_{u\ge r-j} \ \sum_{\vec{v}\in G_{r,j,u}\setminus R_{r,j,u}}
\prod_{k=1}^j\mu_{\theta^{v_{i_k}}\omega}(A_{n+mu_k}).
$$
As there are $\binom{r-1}{j-1}$ many ways to choose the indices $i_1,\dots,i_j$,
we can now write for the last sum on the RHS:
$$
\sum_{\vec{v}\in G_{r,j,u}\setminus R_{r,j,u}} \ 
\prod_{k=1}^j\mu_{\theta^{v_{i_k}}\omega}(A_{n+mu_k})
=\binom{r-1}{j-1}\sum_{u_1+\cdots+u_j=u}\ \sum_{\vec{w}\in G_{j}}
\prod_{k=1}^j\mu_{\theta^{w_i}\omega}(A_{n+mu_k})+E_{r,j,u},
$$
where the error splits into two parts: $E_{r,j,u}=E_{r,j,u}'+E_{r,j,u}''$. 
We now estimate the two last error terms separately, where (a)
$E_{r,j,u}'$
accounts for the over counts of short returns on the RHS and (b)
$E_{r,j,u}''$ accounts for the contribution made by $\vec{v}\in R_{r,j,u}$ 
which are included on the RHS but are excluded on the LHS.

The summation over $u_1,\dots,u_j$ is such that the total overlap $u$ has 
been divided into $r-j$ non-empty sections for the short returns and 
then are clustered into the $j$ clusters where some of the $u_k$ might be 
zero which happens when there is no short return in the associated cluster
(i.e.\ $i_{k+1}=i_k+1$).

(a): The error term $E_{r,j,u}'$ accounts for the over counting of those 
return combinations that do not occur since not all $u_k$ are included in our sums.
For $\vec{v}\in G_r$ the short overlaps cannot be larger than $n$ which limits
the vales and multiplicities for the overlaps $u_k$. On the RHS however we 
impose no such restriction and thus have to correct for it with the error term
$E_{r,j,u}'$. 
Since every overlap $u_k$ has to be generated by at least one of the $r-j$ short returns
each of which in its turn is bounded by $M$ we therefore obtain for 
$E_{r,j}'(n)=\sum_jE_{r,j,u}'$
 the following  upper bound 
\begin{eqnarray*}
|E_{r,j}'(n)|&\le&\sum_{u\ge r-j}\ \sum_{\substack{u_1+\cdots+u_j=u\\ \min_k u_k\ge \frac{M}m}}\frac1{j!} \prod_{k=1}^j\mu_\omega(\zeta_{n,u_k}^x)\\
&\le&\sum_{u\ge\frac{M}m}\ \sum_{u_1+\cdots+u_j=u}\frac1{j!} \prod_{k=1}^j\mu_\omega(\zeta_{n,u_k}^x)\\
&\le&c_1t^j\sum_{u\ge\frac{M}m}\binom{u+j-1}{j-1}\vartheta^u
\end{eqnarray*}
for $n$ large enough. Hence $E_r'(n)=\sum_j\binom{r-1}{j-1}E_{r,j}'\longrightarrow0$
as $n\to\infty$ (as $\frac{M}m>\frac{n}{2m}\to\infty$) for almost every $\omega$
as the sum on the RHS is a tail sum of 
$\sum_{u=0}^\infty\binom{u+j-1}{j-1}\vartheta^u=(1-\vartheta)^{-j}$.

(b) The second part is given by $E_{r,j,u}''=\sum_{\vec{v}\in R_{r,j,u}}\prod_{k=1}^j\mu_{\theta^{w_i}\omega}(A_{n+mu_k})$.
For $n$ large enough we obtain as at the end of the proof of Lemma~\ref{R.small},
\begin{eqnarray*}
|E_r''|
&\le&c_2\sum_{j=2}^r\binom{r-1}{j-1} \sum_{s=1}^{j-1} \binom{j-1}{s-1}
\frac{t^s}{s!}\left(\frac{\delta t}{N}\right)^{j-s} 
\sum_{u=r-j}^\infty\binom{u-1}{r-j-1}\vartheta^u\\
&\leq&c_2\sum_{j=2}^r\binom{r-1}{j-1} \sum_{s=1}^{j-1} \binom{j-1}{s-1}
\frac{t^s}{s!}\left(\frac{\delta t}{N}\right)^{j-s} 
\left(\frac{\vartheta}{1-\vartheta}\right)^{r-j}.
\end{eqnarray*}
Since $j-s\ge1$ we get that $E_r''\to 0$ as $n\to\infty$ provided
$\delta/N\to0$.

\vspace{3mm}

\noindent Combining the estimates for these two error terms we conclude that $
E_r=E_r'+E_r''\to 0$ as $n\to\infty$ provided $\delta/N\to0$.

We thus obtain
\begin{eqnarray*}
S_r(n)&=&\sum_j\binom{r-1}{j-1}(1+\mathcal{O}(j\psi(\delta)))
\sum_{u\ge r-j}\ \sum_{u_1+\cdots+u_j=u}\frac1{j!}
\prod_{k=1}^j\left(\sum_{w=1}^N\mu_{\theta^{w}\omega}(A_{n+mu_k})\right)+E_r,
\end{eqnarray*}
and let $\delta\to\infty$ with $n\to\infty$. 
Hence we obtain for almost every $\omega$  the innermost sum is
$\mu_\omega(\zeta_{n,u_k}^x)\to t\vartheta^{u_k}$ and consequently
get that the principal term converges as follows:
\begin{align*}
S_r(n)\longrightarrow & \sum_j\binom{r-1}{j-1}
\sum_{u\ge r-j}\binom{u-1}{r-j-1}\frac1{j!}t^j\vartheta^u\\
&\qquad=\sum_j\frac{t^j}{j!}\binom{r-1}{j-1}\left(\frac\vartheta{1-\vartheta}\right)^{r-j}\\
&\qquad
=\frac1{(1-\vartheta)^r}\sum_j\frac{(\Theta t)^j}{j!}\binom{r-1}{j-1}\vartheta^{r-j}.
\end{align*}
The  combinatorial factor $\binom{u-1}{r-j-1}$ expresses the number of ways 
in which $u$ overlaps are distributed into $r-j$ short returns $\le M$
(and which are then clustered into $j$ clusters where some of them might be 
empty). This implies
$$
S_r(n)\longrightarrow Q_r(\Theta t,\vartheta)
$$
and thus verifies condition~(a) of Proposition~\ref{sevastyanov}.

To verify assumption~(b) we obtain
$$
\sum_{\vec{v}\in R_r} \mu_\omega(C_{\vec{v}}) \leq K_1\alpha^{r-1}\sum_{j=2}^r\sum_{s=1}^{j-1}
\left(\begin{array}{c}j-1\\s-1\end{array}\right) (\delta\eta^M)^{j-s}\frac{\mu_\omega(\zeta_n^x)^s}{s!}
\left(\begin{array}{c}r-1\\j-1\end{array}\right)(\alpha\eta^m)^{r-j},
$$
where $\alpha=1+\psi(0)$. Hence condition~(II) of Proposition~\ref{sevastyanov} is satisfied
since $j-s\ge1$.
\end{proof}

\section{Examples}\label{examples}

Let us note that  the recurrence properties at periodic points or otherwise do
not require the entropy to be finite (which is 
necessary in the theorem of Shannon-McMillan-Breiman for example). 
To make this point we will give below an example of infinite entropy

\vspace{4mm}

\subsection{Two-element Bernoulli shifts}
Some classes of examples of random SFTs to which our results apply was given in \cite[Section 6]{RT15}, including examples in the infinite alphabet case.  However, to give the reader some idea of systems to which our methods apply, we first give an elementary example (this is a simple version of \cite[Example 19]{RouSauVar13}).

Let $\Omega=\Sigma=\{0,1\}^{\N_0}$ and let $\PP$ be a Gibbs measure on $\Omega$
 (i.e., $\PP$ need not be Bernoulli or Markov). In this case the fibre $\Sigma_\omega$
 is equal to $\Sigma$ for every $\omega$ (in particular the transition matrix
 $A=A(\omega)$ is the full matrix of $1$s. Then fixing  $\alpha, \beta\in (0,1)$, for $\omega=(\omega_0, \omega_1, \ldots)\in \Omega$, let
\begin{equation*}
p(\omega)=\begin{cases} \alpha & \text{ if }  \omega_0=0,\\
\beta & \text{ if }  \omega_0=1.
\end{cases}
\end{equation*}
Then we can define a random Bernoulli measure by
$$\mu_{\omega}[x_0, \ldots, x_n]=p_{x_0}(\omega)p_{x_1}(\theta\omega)\cdots p_{x_n}(\theta^n\omega)$$ where 
\begin{equation*}
p_{x_i}(\omega)=\begin{cases} p(\omega) & \text{ if }  x_i=0,\\
1-p(\omega) & \text{ if }  x_i=1.
\end{cases}
\end{equation*}

As shown in \cite[Example 19]{RouSauVar13}, this system satisfies conditions (i)--(iv), so Theorem~\ref{psi-mixing} holds.  Moreover we can give a formula for the parameter $\vartheta$ explicitly: if $x\in \Sigma$ is a periodic point of period $m$, the Bernoulli property of our sample measures and $\theta$-invariance of $\PP$ allow us to compute that
$$
\vartheta(x)=\int p_{x_0}(\omega)p_{x_1}(\theta\omega)\cdots p_{x_{m-1}}(\theta^{m-1}\omega)~d\PP(\omega).
$$

\subsection{i.i.d.\ infinite alphabet systems}
Let us now consider a stochastic system which is a shift space $\Sigma$ over
a countably infinite alphabet with a Bernoulli measure which we below choose
to obtain infinite entropy. The `driving space' $\Omega$ will be  the infinite product
of intervals.

Let $I$ be a measurable space with a measure $m$ and 
 let $\Omega=I^{\N_0}$ be equipped with the product measure $\PP$. 
 Let $\Sigma=\mathbb{N}^{\N_0}$ with the left shift map $\sigma$ be the full shift over a 
 countable alphabet and $\vec{p}:\Omega\to(0,1)^{\N}$ a function that depends only on 
 the zeroth coordinate, i.e.\ $p(\omega)=p(\omega_0)$, and satisfies
 $\sum_{n=1}^\infty p_n(\omega_0)=1$ for all $\omega_0$,
where $p_n$ are the components of $\vec{p}$. Assume that $\sup_{\omega_0,n}p_n(\omega_0)<1$.
For every $\omega=(\omega_0, \omega_1, \ldots)\in \Omega$ we thus obtain a Bernoulli measure
$\mu_\omega$ on $\Sigma$ defined by
$$
\mu_{\omega}(x_0, \ldots, x_n)=p_{x_0}(\omega)p_{x_1}(\theta\omega)\cdots p_{x_n}(\theta^n\omega).
$$ 
Clearly, $\mu_\omega$ satisfies the Assumptions~(i) and~(iv). The marginal measure $\mu$
is  Bernoulli with weights $\bar{p}_n=\int_{\Omega}p_n(\omega_0)\,d\PP(\omega)$
and consequently $\psi$-mixing. Assumption~(iii) therefore need not be met, so the conclusions of Theorem~\ref{psi-mixing} hold..

To compute $\vartheta$, if $x\in \Sigma$ is a periodic point of minimal period $m$ then, as before,
$$
\vartheta(x)=\int_\Omega p_{x_0}(\omega)p_{x_1}(\theta\omega)\cdots p_{x_{m-1}}(\theta^{m-1}\omega)\,d\PP(\omega)
=\prod_{j=0}^{m-1}\bar{p}_{x_j}
=\mu(x_0x_1\cdots x_{m-1})
$$
by $\theta$-invariance of $\PP$ and since 
$\PP(\omega_0\omega_1\cdots\omega_{m-1})=\prod_{j=0}^{m-1}\PP(\omega_j)$.
Equivalently, if  $N_s=|\{j\in\{0,\dots,m-1\}: x_j=s\}|$ denotes the number of times 
 the symbol $s\in\N$ occurs during a period of $x$, then
$$
\vartheta(x)=\prod_s\bar{p}_s^{N_s}.
$$

\subsubsection{Two-element Bernoulli revisited}
In the special case above when the two element 
alphabet $\{0,1\}$ is used we can equip $\Omega=\{0,1\}^{\N_0}$ with the 
Bernoulli measure with weights $p,q=1-p$, $p,q>0$. Then  $\bar{p}_0=\alpha p+\beta q$
and $\bar{p}_1=(1-\alpha)p+(1-\beta)q$ and consequently 
$$
\vartheta=(\alpha p+\beta q)^\ell((1-\alpha)p+(1-\beta)q)^{m-\ell},
$$
where $\ell $ is the number of times the symbol $0$ occurs on the 
periodic string $x_0\cdots x_{m-1}$.
In the deterministic case when $\alpha=\beta$ we obtain
$\vartheta=\alpha^\ell(1-\alpha)^{m-\ell}$
which is in accordance with Hirata's result.

\subsubsection{Infinite entropy system}
Let us now choose $I=[\varepsilon,1]$ with the normalised Lebesgue measure
and $\varepsilon\in(0,1)$. One can then define a family of 
probability vectors $\{\vec{p}(\omega)\}_{\omega\in \Omega}$ by 
$$
p_n(\omega)=\frac{G(\omega_0)}{n\log^{1+\omega_0}n},
$$
if $n\ge3$, and equal to $0$ if $n=1,2$,
where $G(\omega_0)\in\mathbb{R}_+$ is a normalising constant for every $\omega_0\in I$.
The marginal measure $\mu$ is Bernoulli with weights 
$\bar{p}_n=\int_\varepsilon^1\frac{G(t)}{n\log^{1+t}n}\,\frac{dt}{1-\varepsilon}$ 
where $G(t)\sim\frac1t$. Its entropy is infinite since
\begin{eqnarray*}
h(\mu)&=&-\sum_n\int_\varepsilon^1\frac{G(t)}{n\log^{1+t}n}\,\frac{dt}{1-\varepsilon}
\log\int_\varepsilon^1\frac{G(t)}{n\log^{1+t}n}\,\frac{dt}{1-\varepsilon}\\
&\ge&\sum_n\int_\varepsilon^1\frac{G(t)}{n\log^{1+t}n}\,\frac{dt}{1-\varepsilon}
\left(\log n+\log\log n-\log\int_\varepsilon^1G(t)\,\frac{dt}{1-\varepsilon}\right)\\
&\ge&\sum_n\int_\varepsilon^1\frac{G(t)}{n\log^{t}n}\,\frac{dt}{1-\varepsilon}
-c_1\sum_n\int_\varepsilon^1\frac{G(t)}{n\log^{1+t}n}\,\frac{dt}{1-\varepsilon}
=\infty
\end{eqnarray*}
as the second term converges, where we used that $\frac1{\log^{1+t}n}\le\frac1{\log n}$ and  $c_1=\log\int_\varepsilon^1G(t)\,\frac{dt}{1-\varepsilon}$.

Again, if $x\in\Sigma$ is periodic with minimal period $m$ then
$\vartheta(x)=\bar{p}_{x_0}\bar{p}_{x_1}\cdots\bar{p}_{x_{m-1}}$.

\subsection{Equilibrium states for Axiom A systems} \label{example.Axiom.A}
It is standard to code Axiom A dynamical systems to understand their ergodic properties, see \cite{Bow}.  In this case, a H\"older potential gives rise to a Gibbs state.  Here we briefly discuss the random case, going directly to the symbolic setting, and show how our results apply here.  We emphasise that in our exposition, we fix the dynamics and put the randomness in the measure.

Let $(\Omega, \theta)$ be a two-sided shift with an invariant measure
$\PP$. We assume $(\Sigma, \sigma)$ is a topologically mixing subshift of finite type and $\{f_\omega\}_{\omega\in \Omega}$ 
a family of H\"older continuous functions on $\Sigma$ whose H\"older norms are uniformly bounded.
For some $\hat\kappa\in(0,1)$, the $\hat\kappa$-H\"older norm of a function $f:\Sigma\to\mathbb{R}$ 
is given by $\|f\|=|f|_\infty+\sup_n\hat\kappa^{-n}\var_nf$, where  
$\var_n f=\sup_{x_i=y_i: |i|<n}|f(x)-f(y)|$ is the $n$-variation of $f$.

By~\cite{Kif98} there exist random equilibrium states $\mu_\omega$ that satisfy the 
generalised invariance property $\sigma\mu_\omega=\mu_{\theta\omega}$.
The fibred measures $\mu_\omega$ are Gibbs with respect to $f_\omega$.
 We can assume that the functions $f_\omega$ have zero pressure.
We conclude from the Gibbs property that the fibred measures are $\psi$-mixing 
where $\psi$ decays exponentially at some rate $\kappa<1$.
Recall that (see e.g.~\cite{Bow})  $\mu_\omega=h_\omega\nu_\omega$ where 
$h_\omega$ are the normalised eigenfunction for the 
largest eigenvalue of the transfer operator and $\nu_\omega$ are the associated 
eigenfunctionals which are $e^{-f_\omega}$-conformal, i.e.\ if $\sigma$ is one-to-one
 on a set $A\subset\Sigma$ then  $\nu_\omega(\sigma A)=\int_Ae^{-f_\omega}\,d\nu_\omega(x)$.
If we replace $f$ by the normalised function
 $\tilde{f_\omega}=f_\omega+\log h_\omega-\log h_\omega\circ\sigma$
 then $\mu_\omega$ is $e^{-\tilde{f}_\omega}$-conformal. We now assume that
 $\mu_\omega$ is $e^{-\tilde{f}_\omega}$-conformal for every $\omega$.
 Moreover, if $U(\alpha)$ is the $n$-cylinder which is determined by the 
 $n$-word $\alpha$, then
 $$
 \mu_\omega(U(\alpha))=\int\chi_{U(\alpha)}\,d\mu_\omega
 =\int\mathcal{L}_\omega^n\chi_{U(\alpha)}\,d\mu_\omega
 =\int e^{f_\omega^n(\alpha x)}\,d\mu_\omega(x),
 $$
 where $\mathcal{L}_\omega$ is the transfer operator for  the normalised function $f_\omega$,
 $f_\omega^n=\sum_{j=0}^{n-1}f\circ\sigma$ is the $n$-th ergodic sum of $f_\omega$
 and $\alpha x$ is the concatenation of $\alpha$ and $x$ (subject to the transition
 rules).

Let $\mu=\int\mu_\omega\,d\PP(\omega)$ be the marginal measure. We now want
to establish that $\mu$ is $\psi$-mixing
and for this purpose assume that the family of functions $\{f_\omega\}_{\omega\in \Omega}$ satisfies the following 
additional regularity assumption with respect to $\omega$: Assume there
exists a constant $K$ so that for every $n$:
$$
|f_\omega-f_{\omega'}|_\infty\le K\hat\kappa^n
$$
for $\omega,\omega'\in\Omega$ for which $\omega_i=\omega'_i\forall |i|\le n$.
The supremum norm is over $\Sigma$.

Let $\alpha$ be an $n$-word in $\Sigma$ and denote by $V=U(\alpha)\subset\Sigma$
the associated $n$-cylinder. Similarly for an $m$-word $\beta$ we write $W=U(\beta)$
for the associated $m$-cylinder.
Then 
\begin{eqnarray*}
\mu(V\cap\sigma^{-n-k}W)
&=&\int_\Omega\mu_\omega(V\cap\sigma^{-n-k}W)\,d\PP(\omega)\\
&=&\int_\Omega\mu_\omega(V)\mu_{\theta^{-n-k}\omega}(W)(1+\mathcal{O}(\kappa^k))\,d\PP(\omega)\\
&=&(1+\mathcal{O}(\kappa^k))\int_\Omega
\int_\Sigma e^{f_\omega^n(\alpha x)}\,\mu_\omega(x)\,
\int_\Sigma e^{f_{\theta^{-n-k}\omega}^m(\beta y)}\,d\mu_{\theta^{-n-k}\omega}(y)\,d\PP(\omega).
\end{eqnarray*}
Define $\omega^{(n,k)}(\omega)\in\Omega$ by putting $\omega^{(n,k)}_i=\omega_i$ 
for $i\le n+\frac{k}2$ and $\omega^{(n,k)}_i=\omega_{i-(n+\frac{k}2)}$ for $i>n+\frac{k}2$.
This implies for all $x\in\Sigma$
$$
f_\omega^n(\alpha x)-f_{\omega^{(n,k)}}^n(\alpha x)=\mathcal{O}(\kappa^{\frac{k}2}),
$$
and for all $y$
$$
f_{\theta^{-n-k}\omega}^m(\beta y)-f_{\theta^{-n-k}\omega^{(n,k)}}^m(\beta y)=\mathcal{O}(\kappa^\frac{k}2).
$$
Hence 
\begin{eqnarray*}
\mu(V\cap\sigma^{-n-k}W)\hspace{-3cm}&&\\
&=&(1+\mathcal{O}(\hat\kappa^\frac{k}2))\int_\Omega
\int_\Sigma e^{f_{\omega^{(n,k)}}^n(\alpha x)}\,\mu_\omega(x)\,
\int_\Sigma e^{f_{\theta^{-n-k}\omega^{(n,k)}}^m(\beta y)}\,d\mu_{\theta^{-n-k}\omega}(y)\,d\PP(\omega)\\
&=&(1+\mathcal{O}(\hat\kappa^\frac{k}2))\int_\Omega
\int_\Sigma e^{f_{\omega^{(n,k)}}^n(\alpha x)}\,\mu_\omega(x)\,d\PP(\omega)
\int_\Omega\int_\Sigma e^{f_{\theta^{-n-k}\omega^{(n,k)}}^m(\beta y)}\,d\mu_{\theta^{-n-k}\omega}(y)\,d\PP(\omega).
\end{eqnarray*}
Hence replacing the the modified $\omega^{(n,k)}$ again by $\omega$ we obtain
$$
\int_\Omega
\int_\Sigma e^{f_{\omega^{(n,k)}}^n(\alpha x)}\,\mu_\omega(x)\,d\PP(\omega)
=(1+\mathcal{O}(\hat\kappa^\frac{k}2))\int_\Omega
\int_\Sigma e^{f_{\omega}^n(\alpha x)}\,\mu_\omega(x)\,d\PP(\omega)
=(1+\mathcal{O}(\hat\kappa^\frac{k}2))\mu(U(\alpha)),
$$
and similarly
$$
\int_\Omega\int_\Sigma e^{f_{\theta^{-n-k}\omega^{(n,k)}}^m(\beta y)}\,d\mu_{\theta^{-n-k}\omega}(y)
=(1+\mathcal{O}(\hat\kappa^\frac{k}2))\mu(U(\beta)).
$$
We thus obtain 
$$
\mu(V\cap\sigma^{-n-k}W)=(1+\mathcal{O}(\kappa'^{k}))\mu(V)\mu(W),
$$
that is, the marginal measure $\mu$ is $\psi$-mixing at rate $\kappa'=\max\{\kappa, \sqrt{\hat\kappa}\}$.

Let us note that condition~(iii) is  satisfied since $\mu$ is $\psi$-mixing.

Thus, if $x$ is a periodic point with minimal period $m$, then with $\alpha=x_0\cdots x_{m-1}$,
\begin{eqnarray*}
\mu(A_{n+m}(x))
&=&\int_\Omega\int_\Sigma\chi_{A_{n+m}(x)}\,d\mu_\omega\,d\PP(\omega)\\
&=&\int_\Omega\int_\Sigma e^{f_\omega^m(\alpha y)}\chi_{A_{n}(x)}(y)\,d\mu_\omega(y)\,d\PP(\omega)\\
&=&(1+\mathcal{O}(\kappa^n))
\int_\Omega e^{f_\omega^m(x)}\int_\Sigma \mu_\omega(A_n(x)))\,d\PP(\omega).
\end{eqnarray*}
In particular the limit
$$
\vartheta=\lim_{n\to\infty}\frac{\mu(A_{n+m}(x))}{\mu(A_n(x))}
$$
exists and converges exponentially: 
$\frac{\mu(A_{n+m}(x))}{\mu(A_n(x))}=\vartheta+\mathcal{O}(\kappa^n)$.
We can then choose $\delta(n)$ to be proportional to 
$|\log \mu(A_n(x))|$, or equivalently a multiple of $n$, and obtain the following result.

\begin{theorem} Let $(\Sigma, \sigma)$ be an Axiom A system and
 $\{\mu_\omega\}_{\omega\in \Omega}$ be a family of equilibrium states for H\"older continuous potentials $\{f_\omega\}_{\omega\in \Omega}$
whose H\"older norms are uniformly bounded. We moreover assume that the family of functions
$\{f_\omega\}_{\omega\in \Omega}$ is H\"older continuous in $\omega$.

If $x\in\Sigma$ is periodic with minimal period $m$, then the value 
$\vartheta=\lim_{n\rightarrow\infty}\frac{\mu(A_{n+m}(x))}{\mu(A_n(x))}$ exists. 
Moreover for every parameter value $t>0$ and $r=0,1,\dots$
 one has
$$
\mathbb{P}(\zeta_n^x=r)\to e^{-t}P_r(\Theta t,\vartheta)
$$
as $n\to\infty$ ($\Theta=1-\vartheta$).
\end{theorem}


\end{document}